\documentclass{article}

\usepackage[utf8]{inputenc}
\usepackage{amsmath}
\usepackage[english]{babel}
\usepackage{newlfont}
\usepackage{latexsym}
\usepackage[utf8]{inputenc}
\usepackage{enumerate}
\usepackage{xy}
\usepackage{amsthm,amssymb,amsmath,amsopn,amsfonts,pb-diagram,accents}
\usepackage{cite}
\usepackage{stmaryrd,calc,enumerate,mathrsfs,amscd,wasysym,footnote}
\usepackage{graphicx}
\usepackage{fancyhdr}
\usepackage{indentfirst}
\usepackage{geometry}
\usepackage{mathtools}
\usepackage[normalem]{ulem}
\usepackage{bbm}
\usepackage{float}
\usepackage{hyperref}
\usepackage{comment}
\usepackage{forest}
\usepackage{tikz}
\usetikzlibrary{shapes.geometric,positioning}

\usepackage{wrapfig}
\usepackage{xcolor}
\usepackage{algorithm}
\usepackage[noend]{algpseudocode}
\usepackage[square, sort, numbers]{natbib}
\usepackage{graphicx}
\usepackage{cleveref}
\usepackage{tabularx}
\usepackage{soul}
 \usepackage{titlefoot}

\usepackage[para]{footmisc} 
\usepackage{fixfoot} 
\DeclareFixedFootnote{\repnote}{Dipartimento di Scienze Matematiche “Giuseppe Luigi Lagrange”, Dipartimento di Eccellenza 2018-2022, Politecnico di Torino, Corso Duca degli Abruzzi 24, 10129 Torino Italy.}

\theoremstyle{plain}

\newtheorem{thm}{Theorem}[section]

\newtheorem{cor}[thm]{Corollary}

\newtheorem{prop}[thm]{Proposition}

\theoremstyle{definition}
 
\newtheorem{oss}[thm]{Remark}
\newtheorem{ese}[thm]{Example}
\theoremstyle{remark}

\def\diam{\operatorname{diam}}

\usepackage[affil-it]{authblk}

\author{Matteo Levi$^1$\footnote{m.l.matteolevi@gmail.com}, Federico Santagati$^2$\footnote{federico.santagati@polito.it},  Anita Tabacco$^2$\footnote{anita.tabacco@polito.it} and Maria Vallarino$^2$\footnote{maria.vallarino@polito.it}}
\title{{\bf Poincaré inequalities on graphs}}
\date{\vspace{-5ex}}

\begin{document}

\unmarkedfntext{{\em 2020 Mathematics Subject Classification}:  05C63; 05C05; 05C21; 26D10; 30L15.\newline\hspace*{0.5em}{\em Keywords}: Poincar\'e inequalities; graphs; trees; nondoubling measure.\newline\hspace*{0.5em}\!\!\!$^1$: MaLGa Center - DIBRIS  - Università di Genova, Genoa, Italy. \\ $^2$: Dipartimento di Scienze Matematiche “Giuseppe Luigi Lagrange”, Dipartimento di Eccellenza 2018-2022, Politecnico di Torino, Corso Duca degli Abruzzi 24, 10129 Torino Italy.\newline}

\maketitle

  \begin{abstract}
Every graph of bounded degree endowed with the counting measure satisfies a local version of $L^p$-Poincaré inequality, $ p\in[1,\infty]$. We show that on graphs which are trees the Poincaré constant grows at least exponentially with the radius of balls. On the other hand, we prove that, surprisingly, trees endowed with a flow measure support a global version of $L^p$-Poincaré inequality, despite the fact that they are nondoubling measures of exponential growth.
\end{abstract}

\section{Introduction and notation} 

 The doubling condition and the validity of (some form of) Poincaré inequality have been proved to be natural assumptions to develop analysis on noncompact manifolds, Lie groups, infinite graphs, and more generally on metric spaces. See \cite{HK, HEINONEN} and the references therein for an overview on the topic.

These assumptions, however, seem to be often too restrictive: in many concrete situations, indeed, the two conditions may be verified to hold only locally. On the other hand, the family of locally doubling metric measure spaces supporting a local Poincaré inequality is expected to be wide enough to comprise many of the concrete examples of metric space it is common to work with, and it might be considered folklore that the local version of these assumptions should often be enough to develop (local) analysis.

In this note we focus on infinite graphs. First, we give a simple proof of the fact that every graph of bounded degree endowed with the counting measure (more general, any measure uniformly bounded away from zero and infinity) satisfies a local version of $L^p$-Poincaré inequality, for any $p\in[1,\infty]$. 
 This result is probably folklore, the proof being based on a quite standard reasoning; we include it for the lack of precise references in the literature.
We show that in the case of trees the constant appearing in the Poincaré inequality that we prove, which grows exponentially with the radius, is optimal. In particular, the global Poincaré inequality is not satisfied on  trees with degree strictly larger than $2$ endowed with the counting measure. On the other hand, we will show that flow measures, which are a natural family of  nondoubling measures of at least exponential growth on trees, satisfy the global version of Poincaré inequality on trees, hence proving to be better behaved than the counting measure in this context.
To the best of our knowledge, there are no other examples in the literature of global Poincaré inequalities on metric measure spaces of exponential growth. Our result might pave the way to the study of global $L^p$-Poincaré inequalities on nondoubling metric spaces: as far as we know, weighted global Poincaré inequalities have been considered on some nondoubling setting of polynomial growth (see for example \cite{franchi}). 

Let us now introduce some piece of notation in order to describe more precisely the content of this note. Let $X$ be an infinite, locally finite, connected, and undirected graph. We identify $X$ with its set of vertices and we write $x \sim y$ whenever  $x,y \in X$ are neighbors, namely, when they are connected by an edge. 
 We denote by $\mathrm{deg}(x)$ the number of neighbors of $x$. We say that the graph $X$ has bounded degree $b+1$ if $\mathrm{deg}(x)\leq b+1$, for some $b\geq 1$ and every $x\in X$. A path of length $n\in\mathbb N$ connecting two vertices $x$ and $y$ is a sequence $\{x_0,x_1,\dots,x_n\}\subset X$, with no repeated vertices, such that $x_0=x$, $x_n=y$, and $x_i\sim x_{i+1}$ for every $i=0,\dots,n-1$. The distance $d(x,y)$ is defined as the minimum of the lengths of the paths connecting $x$ and $y$. For any $x\in X$ and $r\ge 0$, the ball of radius $r$ and center $x$ is $B_r(x)=\{ y\in X: \ d(x,y)\leq r\}$.
For every subset $E$ of $X$ the diameter of $E$ is $\diam(E)=\sup\{d(x,y): x,y\in E\}$.

We say that a subset $E$ of $X$ is \emph{quasiconvex} if for every couple of vertices $x,y$ in $E$ there exists a path contained in $E$ connecting $x$ and $y$ of length $\leq 2 \diam(E)$. This is the same as asking that $E$ is quasiconvex as a metric space (see for instance \cite{HK, HeinKosk}) with the metric induced by the ambient space $X$. Examples of sets which are quasiconvex in any graph are all the geodesically convex sets, as well as balls (which, on a general graph, might not be geodesically convex). Observe that on trees, which are connected graphs without cycles, the notion of quasiconvex set coincides with the notion of connected set; indeed, each connected set in a tree is geodesically convex, since every path is a geodesic.

Any nonnegative function $\mu$ on $X$ induces a measure on $X$; with slight abuse of notation, for any subset $E\subseteq X$, we set $\mu(E)=\sum_{x\in E}\mu(x)$. In particular, we denote by $|\cdot |$ the counting measure, i.e., $|E|$ is the cardinality of $E$.
 
A measure $\mu$ is said to be locally doubling if for any $R>0$ there exists a constant $D(R)$ such that for any $0\le r\le R$ \begin{align}\label{locdoubl}
    \mu(B_{2r}(x))\leq D(R)\mu(B_r(x)),\quad \text{for every } x\in X \ .
\end{align}
 
For any $1\leq p <\infty$, we denote by $L^p(X,\mu)$ the space of functions $f:X \to \mathbb{C}$ such that the norm $\|f\|_{L^p(X,\mu)}=\Big(   \sum_{x\in X}|f(x)|^p \mu(x)\Big)^{1/p}$ is finite, and by $L^\infty(X,\mu)$ the space of function such that $\|f\|_{L^\infty(X,\mu)}=\sup_{x\in X}|f(x)|<\infty$.

For every function $f:X \to \mathbb{C}$ we define the length of the gradient of $f$ as the function $|\nabla f|:X\rightarrow \mathbb R$ defined by 
$$|\nabla f|(x)=\sum_{y \sim x} |f(x)-f(y)|, \qquad x\in X.$$
This is a standard notion for analysis on graphs and it plays the role that the upper gradients defined in \cite{HeinKosk} play in analysis on metric spaces. {\color{green}}

 We say that $(X,\mu)$ satisfies a local $L^p$-Poincaré inequality, $p\in [1,\infty]$, if for any $R>0$ there exists a positive constant $P_p(R)$ such that for   any function $f:X\to\mathbb{C}$ and any quasiconvex set $E$ of diameter $0\le 2r\leq R$ it holds
\begin{equation}\label{PpRC}
\|f-f_E\|_{L^p(E,\mu)}\le  P_p(R)r\||\nabla f|\|_{L^p(E,\mu)},
\end{equation}
where $f_E=\frac{1}{\mu(E)}\sum_{x\in E}f(x)\mu(x)$.

In case the constant $P_p(R)$ may be made independent of $R$, we say that $(X,\mu)$ satisfies a global $L^p$-Poincaré inequality. More precisely, $(X,\mu)$ satisfies a global $L^p$-Poincaré inequality, $p\in [1,\infty]$, if there exists a positive constant $P_p$ such that for any function $f:X\to\mathbb{C}$ and any quasiconvex set $E$ of diameter $2r$ it holds
\begin{equation}\label{PpC}
\|f-f_E\|_{L^p(E,\mu)}\le  P_pr\||\nabla f|\|_{L^p(E,\mu)}.
\end{equation}

Notice that when $E$ is a ball, \eqref{PpRC} and  \eqref{PpC} are the standard local and global $L^p$-Poincaré inequalities studied in the literature \cite{CK, delmotte1999, SF, CS, SC2}.

In Section \ref{sec: 2}, we prove an $L^p$-estimate for $f-f_E$ expressed in terms of the $L^p$-norm of $|\nabla f|$ for every function $f$ and every quasiconvex set $E$ on a graph  $X$ endowed with measures positively bounded from below (see Theorem \ref{th: poincaré debole}). 
 As a consequence, we prove a local $L^p$-Poincaré inequality for quasiconvex sets on every infinite graph endowed with a measure positively bounded from below and above (see Corollary \ref{cor: poincare counting measure}). The counting measure is obviously included in this class of measures. 
 
In Section \ref{sec: opt} we discuss the optimality of the results of Section \ref{sec: 2}. First, we prove that the assumption on the quasiconvexity of the set $E$ in Theorem \ref{th: poincaré debole} cannot be weakened by simply assuming that $E$ is connected. Next, we show that also the assumption on the boundedness from below of the measure cannot in general be dropped by exhibiting an appropriate example (see Example \ref{es : 2}). Moreover, we prove that the growth of the constant involved in the local $L^p$-Poincaré inequality, which may be exponential with respect to the radius of the balls, is optimal in the case when $p=1,\infty$ for a suitable class of trees which includes the homogeneous tree. It is worth mentioning that a similar discussion on the exponential growth of the constant was carried out on the so called $ax+b$ groups in \cite{BPV}.
\\ \indent Surprisingly, in the last section we are able to prove a global $L^p$-Poincaré inequality for quasiconvex sets and for flow measures on infinite trees, a class of measures introduced in \cite{LSTV} (see \eqref{flussi} for the precise definition). This represents a further evidence that these measures, despite being nondoubling, of exponential growth and not positively bounded from below nor from above, are very well behaved with respect to analysis on trees.

We remark that the Poincaré inequalities that we prove, and to which we address simply as $L^p$-Poincaré inequalities, are indeed $(p,p)$-strong Poincaré inequalities. The term strong here refers to the fact that the integral on the right-hand side of \eqref{PpRC} is taken on the same set than the integral on the left-hand side, and not on an enlarged set. The specification $(p,p)$, instead, denotes a difference with another class of inequalities, the $(1,p)$-Poincaré inequalities (see for example \cite[Equation (8.1.1)]{HKST}). These are, for instance, the Poincaré inequalities treated in \cite{HK,HeinKosk} in the generality of metric measure spaces. We point out that the $L^1$-Poincaré inequalities that we prove imply the corresponding $(1,p)$-Poincaré inequalities for every $p>1$, while in general they are not enough to imply $L^p$-Poincaré inequalities for any $p>1$. Nevertheless, as previously described, for the cases under study, we obtain $L^p$-Poincaré inequalities for every $p\in [1,\infty]$.

 (Local) Poincaré inequalities combined with the (local) doubling condition, are a standard tool to obtain (local) Harnack inequalities both in continuous and discrete settings (see \cite{delmotte1997, delmotte1999, RSV, SF}). In particular, it would be interesting to combine the global Poincaré inequalities for trees endowed with nondoubling flow measures that we obtain in Section \ref{sec: flows} with a suitable substitute of the doubling condition for a family of connected subsets of the tree (see \cite{LSTV}) to get global Harnack inequalities and estimates of the heat kernel of flow Laplacians.
This seems to be a challenging problem and it will be object of further investigation.


Along the paper, 
we use the standard notation $f_1(x)\lesssim f_2(x)$ to indicate
that there exists a positive constant $C$, independent of the variable $x$ but possibly depending on
some involved parameters, such that $f_1(x) \le Cf_2(x)$ for every $x$. When both $f_1(x)\lesssim f_2(x)$ and $f_2(x) \lesssim f_1(x)$ are valid, we will write $f_1(x) \approx f_2(x).$

\section{Bounded measures on graphs} \label{sec: 2}
Let $X$ be an infinite, locally finite, connected, and undirected graph.
For every $\alpha>0$ we denote by $\mathcal{M}_\alpha$ the class of measures positively bounded from below by $\alpha$, namely, 
$\mu \in \mathcal{M}_\alpha$ if $\mu(x) \ge \alpha$ for every $x \in X.$ We underline that the counting measure belongs to $\mathcal{M}_1$.

 The proof of next theorem is based on the following standard observation (see for instance, for the case $p=2$, \cite{CS} or \cite[Proposition 8.3.1]{lenz}): for every $x,y\in X$ and every $p\in [1,\infty )$,
\begin{equation}\label{standard trick}
    |f(x)-f(y)|^p \le |\gamma_{xy}|^{p-1}\sum_{z\in \gamma_{xy}}|\nabla f|(z)^p,
\end{equation}
where $\gamma_{xy}$ is a path connecting $x$ and $y$ and $|\gamma_{xy}|$ its length. Indeed, let $\gamma_{xy}=\{x_i\}_{i=0}^n$, where $x_0=x$, $x_n=y$ and $x_i \sim x_{i+1}$, $i=0,...,n-1.$ Then,
\begin{equation*}
     |f(x)-f(y)| \le \sum_{i=0}^{n-1} |f(x_i)-f(x_{i+1})| \le \sum_{z \in [x,y]}|\nabla f|(z),
\end{equation*}
and \eqref{standard trick} follows by  Hölder's inequality.

\begin{thm}\label{th: poincaré debole}
Let $E\subset X$ be a finite quasiconvex set with $\diam(E)=2r$, $p\in[1,\infty]$, and $f$ be any function on $X$. Then, for any $\alpha>0$ and for every $\mu \in \mathcal{M}_\alpha$,
\begin{align}\label{LpEE}
\|f-f_E\|_{L^p(E,\mu)}\le \bigg(\frac{\mu(E)}{\alpha}\bigg)^{1/p}(4r)^{1-1/p}\||\nabla f|\|_{L^p(E,\mu)}.
\end{align}
\end{thm}

\begin{proof}
Since $E$ is quasiconvex, for $x,y\in E$ we can choose the path $\gamma_{xy}$ in \eqref{standard trick} to have length smaller or equal than  $2 \diam(E)$. Then, for $p\in [1,\infty)$ and $x\in E$, by Jensen's inequality and \eqref{standard trick} one has
\begin{equation}\label{inequality}
\begin{split}
    |f(x)-f_E|^p&\le \Big(\frac{1}{\mu{(E)}}\sum_{y \in E} |f(x)-f(y)| \ \mu(y)\Big)^p \le \frac{1}{\mu{(E)}}\sum_{y \in E} |f(x)-f(y)|^p \ \mu(y)\\
    &\leq \frac{(4r)^{p-1}}{\mu{(E)}}\sum_{y \in E} \sum_{z\in \gamma_{xy}}|\nabla f|(z)^p \ \mu(y).
\end{split}
\end{equation}
 In particular, for every $x\in X$ we have
\begin{equation*}
   |f(x)-f_E|\le \frac{1}{\mu(E)} \sum_{y \in E} \sum_{z\in \gamma_{xy}}\||\nabla f|\|_{L^\infty (E,\mu)} \ \mu(y) \le 4r\||\nabla f|\|_{L^\infty (E,\mu)},
\end{equation*}
which proves the theorem for $p=\infty$. For $p\in [1,\infty)$, instead, one obtains from \eqref{inequality} that
\begin{equation*}
    |f(x)-f_E|^p\le (4r)^{p-1}\sum_{z\in E}|\nabla f|(z)^p.
\end{equation*}
Since $\mu(z)\geq \alpha$ for any $z\in X$, we obtain 
\begin{equation*}
      \sum_{x \in E}|f(x)-f_E|^p \mu(x)\leq \frac{\mu(E)}{\alpha}(4r)^{p-1}\sum_{z \in E}|\nabla f|(z)^p\mu(z).
\end{equation*}
\end{proof}

\begin{oss}
Analyzing the proof of Theorem  \ref{th: poincaré debole}, it is not difficult to observe that for $p=1$ the assumption on the quasiconvexity of the set $E$ may be replaced by the weaker request that $E$ is connected. 
On the other hand, we will show in Section \ref{sec: opt} that for $p\in (1,\infty]$ the assumption that $E$ is connected is not sufficient for Theorem \ref{th: poincaré debole} to hold. Similarly, when $p=\infty$, it is not necessary to require the positive boundedness from below of the measure. Nevertheless, in the next section we will prove that such request cannot be dropped when $p<\infty.$ \qed
\end{oss}

For every finite $\beta>0$, we denote by $\mathcal{M}^\beta$ the class of positive measures bounded from above by $\beta$, namely, $\mu \in \mathcal{M}^\beta$ if $\mu(x) \le \beta$ for every $x \in X$. For every $0< \alpha \le \beta$, we set $\mathcal{M}_\alpha^\beta=\mathcal{M}_\alpha \cap\mathcal{M}^\beta$. It is worth mentioning that $\mathcal{M}_1^1$ uniquely consists of the counting measure.

 We remark that if the graph $X$ has bounded degree, then measures in $\mathcal{M}_\alpha^\beta$ are locally doubling. Indeed, suppose that $\deg(x)\leq b+1$ and let $\mu\in\mathcal{M}_\alpha^\beta$. Then, for every $x \in X$ and $r \ge 0$, $|B_r(x)|\leq 3b^r$ and therefore, for every $R \ge r$,
\begin{equation*}
\frac{\mu (B_{2r}(x))}{\mu (B_r(x))}\leq \frac{\beta}{\alpha}\frac{|B_{2r}(x)|}{|B_r(x)|}\leq \frac{\beta}{\alpha}|B_{2r}(x)|\le\frac{\beta}{\alpha}3b^{2R}=:D(R).
\end{equation*}
On the other hand, graphs which have no bounded degree cannot support locally doubling measures. To see this, suppose that $\mu$ is a positive locally doubling measure on a graph $X$. For any $x \in X$,
\begin{align*}
       D(1/2)\mu(x) \ge \mu(B_1(x)) \ge \sum_{y\sim x} \mu(y) \ge \deg(x) \mu(\overline{y}),    
\end{align*}
 where $\overline{y}\sim x$ is such that $\mu(\overline{y}) \le \mu(y)$ for every $y \sim x.$ By exploiting the locally doubling condition again, 
 \begin{align*}
     \mu(\overline{y}) \ge \frac{\mu(B_1(\overline{y}))}{D(1/2)}\ge \frac{\mu(x)}{D(1/2)}.
 \end{align*} Combining the above inequalities, it follows that for every $x \in X$,
 \begin{align*}
     \mathrm{deg}(x) \le D(1/2)^2.
 \end{align*}

The following corollary shows that for measures in the class $\mathcal{M}_\alpha^\beta$ 
on graphs of bounded degree, we have a local $L^p$-Poincar\'e inequality for quasiconvex sets.

\begin{cor}\label{cor: poincare counting measure}
Suppose that $X$ has bounded degree $b+1$. Fix $R>0$ and let $E\subset X$ be a quasiconvex set with $\diam(E)=2r\leq R$, $p\in[1,\infty]$, and $f$ be any function on $X$. Then, for every $0<\alpha \le \beta<\infty$ and $\mu \in \mathcal{M}_\alpha^\beta$, $(X,\mu)$ satisfies the $L^p$-Poincaré inequality \eqref{PpRC}, i.e., 

\begin{equation*}
\|f-f_E\|_{L^p(E,\mu)}\le P_p(R)r\||\nabla f|\|_{L^p(E,\mu)},
\end{equation*}
with $\displaystyle P_p(R)=4\bigg(3\frac{\beta b^{R}}{4\alpha}\bigg)^{1/p}.$
\end{cor}
\begin{proof}

If $x\in E$, then $E\subseteq B_R(x)$, so that
\begin{equation*}
    \mu(E)\leq\mu(B_R(x)) \le 3\beta b^{R}. 
\end{equation*}
If $r<1$ the result is trivial, so we can suppose $r\geq 1$. Then the result directly follows from Theorem \ref{th: poincaré debole}.
\end{proof}
\begin{oss}
We remark that, for $p=\infty$, the conclusion of Corollary \ref{cor: poincare counting measure} coincides with that of Theorem \ref{th: poincaré debole}. 
 In particular, in this case, there is no actual need to assume any boundedness of the measure, nor bounded degree of the graph; the local $L^\infty$-Poincaré inequality for quasiconvex sets holds for any graph and any measure.  
\\ \indent We mention that a global inequality related to \eqref{PpRC} was obtained in \cite{CS} for $\alpha=\beta=1$ in the case $p=1,2$, under the additional assumption that the counting measure is globally doubling. \qed \end{oss}


\section{Optimality of Theorem \ref{th: poincaré debole}}\label{sec: opt}

The scope of this section is to discuss the optimality of Theorem \ref{th: poincaré debole} under different aspects. First we will prove by means of an example that for $p>1$, in general,
the assumption of quasiconvexity on $E$ cannot be replaced by the weaker assumption of being connected. Next, we prove that for measures which are not positively bounded from below the local $L^p$-Poincaré inequality may fail. Finally, we prove that for $p=1, \infty$ the constant in formula \eqref{LpEE} is optimal when $E$ is a ball, $X$ is a tree and $\mu\in \mathcal{M_\alpha^\beta}$.

\begin{ese}\label{es:quasiconvex}
Referring to Figure \ref{figure}, consider the infinite connected graph $X$ with vertex set labelled by $\mathbb{N}^2$ and the following proximity rule: $(j_1,k_1)\sim(j_2,k_2)$ if and only if $(i)$ $|j_1-j_2|+|k_1-k_2|=1$ or $(ii)$ $k_1=k_2$ odd, $j_1,j_2\leq k_1$ and $|j_1-j_2|\geq 4$.



For any $E \subset X$ and $p \ge 1$, let $\|\cdot\|_{L^p(E)}$ denote $\|\cdot \|_{L^p(E, |\cdot|)}$. Define the sequence of sets $E_k=\{(j,k)\in X: j\leq k \}$, $k\in 2\mathbb{N}$. It is clear that $E_k$ is connected, $1\leq\diam(E_k)\leq 3$ and $|E_k|=k+1.$

\begin{figure}[h]
\begin{center}
\begin{tikzpicture}
    [
        dot/.style={circle,draw=black, fill,inner sep=1pt},
    ]

\foreach \x in {0,...,8}{
    \foreach \y in {0,...,8}{
        \node[dot] at (\x,\y){ };
         \draw[thick] (0,\y) -- (8.5,\y);
         \draw[thick] (\x,0) -- (\x,8.5);
    }
}

\foreach \x in {1,...,8}
    \draw (\x,.1) -- node[below,yshift=-1mm] {\x} (\x,-.1);
\foreach \y in {1,...,8}
    \draw (.1,\y) -- node[xshift=-4mm] {\y} (-.1,\y);
\node[below,xshift=-2mm,yshift=-1mm] at (0,0) {0};

\draw[->,thick,-latex] (0,0) -- (0,9) node[left] {$k$};
\draw[->,thick,-latex] (0,0) -- (9,0) node[below] {$j$};

\draw[thick, out=90,in=90,in looseness=0.5,out looseness=0.5,-]  (3,7) to  (7,7);
\draw[thick, out=90,in=90,in looseness=0.5,out looseness=0.5,-]  (2,7) to  (7,7);
\draw[thick, out=90,in=90,in looseness=0.5,out looseness=0.5,-]  (2,7) to  (6,7);
\draw[thick, out=90,in=90,in looseness=0.5,out looseness=0.5,-]  (1,7) to  (7,7);
\draw[thick, out=90,in=90,in looseness=0.5,out looseness=0.5,-]  (1,7) to  (5,7);
\draw[thick, out=90,in=90,in looseness=0.5,out looseness=0.5,-]  (0,7) to  (7,7);
\draw[thick, out=90,in=90,in looseness=0.5,out looseness=0.5,-]  (1,7) to  (6,7);
\draw[thick, out=90,in=90,in looseness=0.5,out looseness=0.5,-]  (0,7) to  (5,7);
\draw[thick, out=90,in=90,in looseness=0.5,out looseness=0.5,-]  (0,7) to  (4,7);
\draw[thick, out=90,in=90,in looseness=0.5,out looseness=0.5,-]  (0,7) to  (6,7);
\draw[thick, out=90,in=90,in looseness=0.5,out looseness=0.5,-]  (0,5) to  (4,5);
\draw[thick, out=90,in=90,in looseness=0.5,out looseness=0.5,-]  (0,5) to  (5,5);
\draw[thick, out=90,in=90,in looseness=0.5,out looseness=0.5,-]  (1,5) to  (5,5);
\end{tikzpicture}

\end{center}
\caption{A portion of the graph $X$. Proximity rule $(i)$ defines the grid, while $(ii)$ introduces the curved edges.\footnotesize}
\label{figure}
\end{figure}
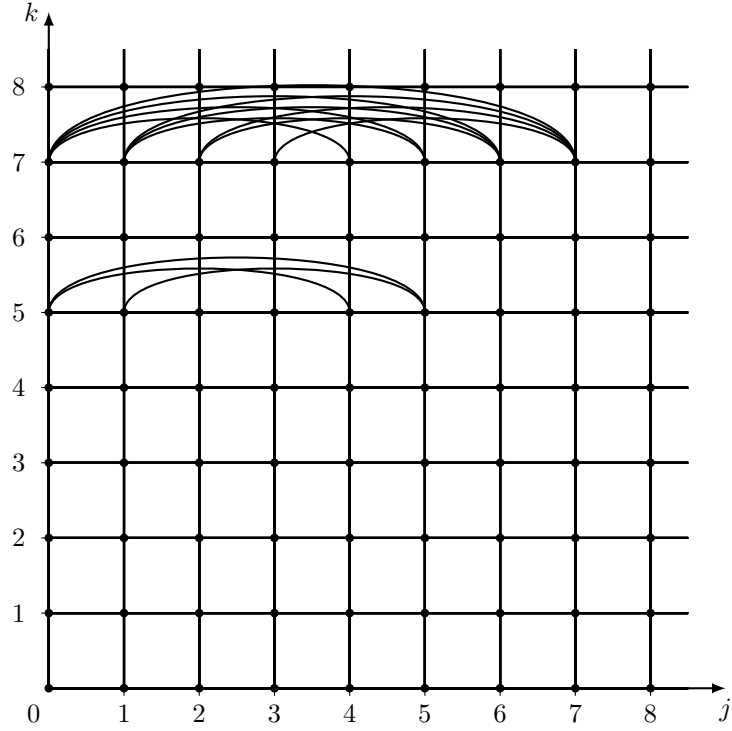
\noindent Consider the function $f:X\to\mathbb{C}$ such that $f(j,k)=j$, for every $(j,k)\in X$. The average of $f$ on $E_k$, with respect to the counting measure, is equal to
\begin{equation*}
   f_{E_k}=\frac{1}{k+1} \sum_{(j,k)\in E_k}f(j,k)= \frac{1}{k+1}\sum_{j=0}^k j=\frac k2.
\end{equation*}

\noindent It follows that, for $1<p<\infty$,
\begin{equation*}
\| f-f_{E_k}\|_{L^p(E_k)}^p=\sum_{j=0}^k \bigg|j-\frac k2\bigg|^p=2 \sum_{j=0}^{k/2} j^p \approx k^{p+1}, 
\end{equation*}
and
\begin{equation*}
\| f-f_{E_k}\|_{L^\infty(E_k)}=\max_{j=0,..,k} \bigg|j-\frac k2\bigg|=\frac k2.
\end{equation*}

\noindent Moreover, for $(j,k)\in E_k$, $|\nabla f|(j,k)=|j-(j-1)|+|j-(j+1)|=2$ if $j\neq 0$, and $\nabla f(0,k)=1$. 
Therefore,  $\| |\nabla f|\|_{L^\infty(E_k)}=2$ and for $1<p<\infty$,
\begin{equation*}
    \| |\nabla f|\|_{L^p(E_k)}^p\approx \sum_{j=0}^k 1= k+1.
\end{equation*}

\noindent Altogether, for any $1<p\leq \infty$,
\begin{equation*}
    \frac{\| f-f_{E_k}\|_{L^p(E_k)}}{|E_k|^{1/p}\diam{(E_k)}^{1-1/p} \| |\nabla f|\|_{L^p(E_k)}}\approx k^{1-1/p}\to \infty, \quad \text{as } k\to\infty,
\end{equation*}
contradicting \eqref{LpEE}. \qed
\end{ese}
We now exhibit a graph endowed with a measure which is not positively bounded from below on which the $L^p$-Poincaré inequality fails for $1\leq p<\infty$.
\begin{ese}\label{es : 2} Let $X=\mathbb Z$ endowed with the measure $$\mu(j)= \begin{cases} \displaystyle\frac{1}{|j|^{}} &\text{if $j \ne 0$,} \\ 
1 &\text{otherwise.} 
\end{cases}$$
Consider the sequence of sets $E_k=[-k,k]\cap \mathbb Z$ and define $f(j)=j$. It is clear that $f_{E_k}=0$ for any $k \ge 1.$ Moreover, $\mu(E_k) \approx \log k$  and $\mathrm{diam}(E_k)=2k+1$ for any $k \ge 2$. A simple computation shows that for any $1 \le p<\infty,$ 
\begin{align*}
   \|f-f_{E_k}\|_{L^p(E_k,\mu)}=\|f\|_{L^p(E_k,\mu)}= \bigg(\sum_{j=-k}^k |j|^{p-1}\bigg)^{1/p} \approx k.
\end{align*} Moreover,
\begin{align*}
    \||\nabla f|\|_{L^p(E_k,\mu)} \approx \bigg(\sum_{j=1}^k \frac{1}{|j|}\bigg)^{1/p} \approx (\log k)^{1/p},
\end{align*} for any $k \ge 2.$ It follows that 
\begin{align*}
    \frac{\|f-f_{E_k}\|_{L^p(E_k,\mu)}}{ \mu(E_k)^{1/p}\mathrm{diam}(E_k)^{1-1/p}\||\nabla f|\|_{L^p(E_k,\mu)}} \approx \frac{k^{1/p}}{(\log k)^{2/p}} \to \infty, \qquad \mathrm{as} \ k \to \infty. 
\end{align*} \qed
\end{ese}
We now discuss the optimality of the constant in Theorem \ref{th: poincaré debole} on trees.
Let $T$ be a tree such that $\mathrm{deg}(x) \ge 2$ for every $x \in T$. 
For any couple of points $x_0\sim y_0\in T$, we define the \textit{triangle} of height $r \in \mathbb N$ and root $[x_0,y_0]$ to be the set $T_0=\{x \in T: \  d(x,x_0)=d(x,y_0)-1\leq r\}.$ The base of $T_0$ is intended to be the set of points of $T_0$ at distance $r$ from $x_0$.

\begin{prop}\label{propPopt} Let $T$ be a tree such that  $2\le \mathrm{deg}(x)\le b+1$ for every $x\in T$,  $0<\alpha\le\beta<\infty$ and $\mu\in\mathcal{M}_\alpha^\beta$. Then, for every ball $B \subset T$ with $\mathrm{diam}(B)=2r$ and $p \in [1,\infty]$, there exists a function $f$ such that 
\begin{align}\label{exponentialgrowth}
    \frac{\| f-f_B\|_{L^p(B,\mu)}}{ \||\nabla f|\|_{L^p(B,\mu)}} \gtrsim C_p(B,\mu),
\end{align} where $C_p(B,\mu)=\mu(B)^{1/p}$ for every $p \in [1,\infty)$ and $C_\infty(B,\mu)=r.$ \end{prop} 
\begin{proof}
  Let $B=B_r(x_0) \subset T$.
Consider two points $x_1,x_2\sim x_0$ and let $T_1,T_2$ be the two disjoint triangles of height $r-1$ and roots, respectively, $[x_1,x_0], [x_2,x_0]$. Clearly $T_1,T_2\subset B$. Define $f$ on $B$ as
\begin{align*}
   f(x)= \begin{cases} d(x,x_0) &\text{if $x \in T_1$}, \\ 
   -Cd(x,x_0) &\text{if $x \in T_2$}, \\ 
   0 &\text{otherwise},
\end{cases} 
\end{align*} and extend $f$ on $T$ by imposing $f(x)=f(y)$ if $x \sim y$ on the remaining vertices.
We choose
\begin{equation*}
    C=\frac{\sum_{x \in T_1}d(x,x_0)\mu(x)}{\sum_{x \in T_2}d(x,x_0)\mu(x)},
\end{equation*}
so that $f_B=0$. Observe that $\|f\|_{L^\infty(B,\mu)}=r\max\{1,C\}$ and $\||\nabla f|\|_{L^\infty(B,\mu)} \le (b+1) \max\{1,C\}.$ Thus, \begin{align*}
    \frac{\|f\|_{L^\infty(B,\mu)}}{\||\nabla f|\|_{L^\infty(B,\mu)}} \ge  \frac{r}{b+1},
\end{align*}
which is \eqref{exponentialgrowth} for $p=\infty.$

We now focus on the case $p\in [1,\infty)$. We claim that for every triangle $T_0$ there exists a triangle $T' \subset T_0$, whose base is contained in the base of $T_0$, such that $ \mu(T')\approx \mu(T_0\setminus{T'})$, with constants depending only on $\alpha, \beta$ and $b$. If the above claim holds, we are done. 

Indeed,  it is easy to see that $B \setminus \{x_0\}$ can be decomposed  as the union of at most $b+1$ disjoint triangles $\{T_i\}_{i=1}^n$ with roots $[x_i, x_0]$, where $x_i\sim x_0$, and height $r-1$. Since $\mu(x_0)\leq \beta$ and $n\leq b+1$, it is clear that there exists at least one triangle $T_j$ among them  with measure $\mu(T_j)\gtrsim \mu(B)$. Now, by the aforementioned claim, we can choose a triangle $T'\subset T_j$ with root $[x',y']$ such that $\mu(T')\approx \mu(T_j\setminus T').$ Clearly $\mu(T')\approx \mu(T_j) \approx \mu(B).$ We conclude by defining $f$ on $B$ by
\begin{align*}
    f(x)= \begin{cases}1 &{x \in T'}, \\
    -\mu(T')/\mu(T_j\setminus T') &{x \in T_j\setminus T'}, \\ 
    0&\text{otherwise;}
    \end{cases} 
\end{align*} and we extend $f$ on $T$ by defining $f(x)=f(y)$ if $x \sim y$ on the remaining vertices.
It is obvious that $f_B=0$, and 
\begin{equation*}
    \sum_{x \in B}|f(x)|^p\mu(x)  \approx \mu(B).
\end{equation*}
Moreover, $|\nabla f|(x)=0$ unless $x =x_j, x_0, x',y'$, in which case
\begin{align*}
    |\nabla f|(x) \approx 1+ \frac{\mu(T')}{\mu(T_j\setminus T')} \approx 1.
\end{align*} It follows that for every $p \in [1,\infty)$
\begin{align*}
    \frac{\|f-f_B\|_{L^p(B,\mu)}}{\||\nabla f|\|_{L^p(B,\mu)}} \gtrsim \mu(B)^{1/p},
\end{align*} 
which is inequality \eqref{exponentialgrowth}.
\\ \indent It remains to prove the claim. Let $T_0$ be a triangle of height $r$ and, for every integer $n\in [1,r]$, let $T_n$ be the triangle of height $r-n$ of maximal measure among those contained in $T_{n-1}$. 
Let $n$ be the minimum integer for which
\begin{align*}
     \frac{\mu(T_n)}{\mu(T_0\setminus T_n)} \le 2.
\end{align*}
We can assume that $r$ is large enough so that the above inequality is actually satisfied for some $n\in [1,r]$ (indeed, if $r$ is small there is nothing to prove).
We have $\mu(T_n)\leq 2 \mu(T_0\setminus T_n)$ and, on the other hand,
\begin{align*}
    \mu(T_0 \setminus T_n)&\le \mu(T_0 \setminus T_{n-1})+b\mu(T_n)+\beta \\ &< \frac{1}{2}\mu(T_{n-1})+b\mu(T_n)+\beta \\ &\le \frac{1}{2}(b\mu(T_n)+\beta)+b\mu(T_n)+\beta \le \frac{3}{2}\bigg(b+\frac{\beta}{\alpha}\bigg)\mu(T_n).
\end{align*}
The claim is proved and the proof is completed.
\end{proof}

We underline that when $\mathrm{deg}(x) \ge 3$ the term $C_p(B,\mu)=\mu(B)^{1/p}$ has exponential growth with respect to the radius of $B$ since $\mu(B) \ge \alpha 2^r$ if $B=B_r(x_0)$ for some $x_0 \in T$ and $r \in \mathbb N.$ \\  

We now apply the previous proposition in order to deduce an optimal Poincaré inequality for $p=1$ and $p=\infty$ on a suitable class of trees, which includes the homogeneous tree endowed with the counting measure.

 \begin{thm}\label{opt} Let $T$ be a tree such that  $2\le \mathrm{deg}(x)\le b+1$ for every $x\in T$. Fix $0<\alpha\le\beta<\infty$ and let $\mu \in \mathcal{M}_\alpha^\beta$ be a measure on $T$ such that, for every ball $B\subset T$ with \textrm{diam}$(B)=R$, $\mu(B) \approx h(R)$ where $h: \mathbb{N} \to \mathbb R$ is a given function. 
Then, if $p\in [1,\infty]$, the following inequalities hold 
\begin{align}\label{poincareabove}
   \| f-f_B\|_{L^p(B,\mu)} \lesssim  h(R)^{1/p} R^{1-1/p} \||\nabla f|\|_{L^p(B,\mu)}.
\end{align} Moreover, if $p=1,\infty$, the previous inequalities are optimal, i.e., there exists a function $g$ such that
\begin{align}\label{poincarébelow}
         \frac{\| g-g_B\|_{L^p(B,\mu)}}{\||\nabla g|\|_{L^p(B,\mu)}} \gtrsim h(R)^{1/p} R^{1-1/p}, \qquad p=1,\infty.
\end{align}
 \end{thm}
\begin{proof} Since $\mu(B) \approx h(R)$, Theorem \ref{th: poincaré debole} implies $\eqref{poincareabove}$  and, in the case $p=1,\infty$, Proposition \ref{propPopt} yields \eqref{poincarébelow}. 
\end{proof}

\section{Flow measures on trees}\label{sec: flows}
 In this section we prove a global $L^p$-Poincaré inequality on connected sets for a tree endowed with a so called flow measure, to be defined soon. What is remarkable here, with respect to the results of Section \ref{sec: 2}, is that we are able to promote the inequality from local to global, despite the fact that the measure is not required to be bounded above nor below, and in principle may not even be locally doubling. Moreover, it is not required for the tree to have bounded degree.

Let us briefly introduce the setting and discuss the above comments. We let $T$ be a tree such that $\mathrm{deg}(x) \ge 2$ for every $x \in T$. Fix a point $o\in T$, which we call the origin, and a half-infinite geodesic $o=x_0,x_1,x_2,\dots$. We denote by $\ell(x)$ the level of the vertex $x$, which is defined by $\ell(x)=\lim_{n \to \infty} (n-d(x,x_n))$. For each vertex $x$, let $p(x)$ be its only neighbor such that $\ell(p(x))>\ell(x)$ and let $s(x)$ be the set of the remaining neighbors. We define a partial order relation on $T$ according to which $x\geq y$ if and only if $y$ is closer to $x$ than to $p(x)$.

In this context, we say that $\mu : T \to \mathbb{R}^+$ is a flow measure if
\begin{align}\label{flussi}
    \mu(x)=\sum_{y \in s(x)} \mu(y), \qquad \text{for every } x \in T.
\end{align}

 Observe that flow measures are not necessarily in the class $\mathcal{M_\alpha^\beta}$. Indeed, the most standard example of flow on a homogeneous tree, that is, on a tree where $\deg(x)=q$ for some integer $q\geq 2$ and every vertex $x$, is the \textit{canonical flow}, $\mu(x)=q^{\ell(x)}$, which can be arbitrarily large as well as arbitrarily close to zero. It is also clear from this example that flows are typically nondoubling, since the canonical flow measure of a ball centred at $o$ of radius $r \in \mathbb N$ on the homogeneous tree is approximately $q^r$. It is also possible to construct flow measures which are not even locally doubling; indeed every flow  measure (on any tree) such that the ratio $\mu(x)/\mu(y)$, with $y\in s(x)$, is not bounded above and below uniformly on the tree is not locally doubling. For a proof of this fact and a more thorough discussion on flow measures and their properties we refer the reader to \cite[Proposition 2.2]{LSTV}. 

The conservation property \eqref{flussi} characterizing flows
is equivalent to Kirchhoff’s current law: the total current received by a vertex must
equal the total current released by the vertex. Flows have remarkable properties from the harmonic analysis point of view. Indeed, in \cite{HS} the authors develop a Calderón--Zygmund theory on a homogeneous tree endowed with the canonical flow and, in \cite{LSTV}, this theory is adapted to general trees endowed with any locally doubling flow measure.

We define the difference operator acting on functions $f : T \to \mathbb{C}$ as
\begin{align*}
    d f(x)=f(x)-f(p(x)).
\end{align*}
Observe that for any $f : T \to \mathbb C$ and $x \in T,$ $|df(x)| \le |\nabla f|(x)$.

\begin{thm}\label{theorem: flow} Let $E \subset T$ be  a connected set with $\mathrm{diam}(E)=2r$,  $p \in [1,\infty]$ and $f$ any function on $T$. Then, for every flow measure $\mu$, $(T,\mu)$ satisfies the $L^p$-Poincaré inequality \eqref{PpC} with $P_p=4$, i.e.,
\begin{align*}
    \|f-f_E\|_{L^p(E,\mu)} \le 4r \||\nabla f|\|_{L^p(E,\mu)}.
\end{align*}
\end{thm}
\begin{proof} Let $E\subset T$ be a finite connected set with $\mathrm{diam}(E)=2r$. It is easy to see that
\begin{align}\label{4.2}
    \sup_{x \in E}|\{z \in E : z \ge x\}| \le 2r.
\end{align}
Denote by $x_E$ the vertex with maximum level in $E$. Then, we have that 
\begin{align*}
    |f(x)-f_E| &\le \sum_{y \in E} \bigg( \sum_{x_E \ge z \ge x}|df(z)|+\sum_{x_E \ge z \ge y}|df(z)|\bigg)\frac{\mu(y)}{\mu(E)} \\ &\le 2 \|df\|_{L^\infty(E,\mu)}\sup_{x \in E}|\{z \in E : z \ge x\} \\ & \le 4r \|df\|_{L^\infty(E,\mu)}.
\end{align*} Passing to the supremum and using that $|df|\le |\nabla f|$, we get the desired inequality when $p=\infty.$ \\ 
Assume now $p\in[1,\infty).$ By applying Jensen's inequality, we get that
\begin{align*}
    \sum_{x \in E} |f(x)-f_E|^p\mu(x) &= \sum_{x \in E}  \bigg|\sum_{y \in E} (f(x)-f(y))\frac{\mu(y)}{\mu(E)}\bigg|^p\mu(x) \\ &\le \sum_{x \in E} \sum_{y \in E} |f(x)-f(y)|^p\frac{\mu(y)}{\mu(E)}\mu(x)\\ 
    &\le \sum_{x \in E} \sum_{y \in E} \bigg(\sum_{x_E \ge z \ge x}|df(z)|+\sum_{x_E \ge z \ge y}|df(z)|\bigg)^p\frac{\mu(y)}{\mu(E)}\mu(x).
    \end{align*}
    Then, since $(a+b)^p \le 2^{p-1}(a^p+b^p)$ for any $a,b \ge0$, by Hölder's inequality, \eqref{4.2} and Fubini's Theorem we obtain
    \begin{align*} 
    &\sum_{x \in E} \sum_{y \in E} \bigg(\sum_{x_E \ge z \ge x}|df(z)|+\sum_{x_E \ge z \ge y}|df(z)|\bigg)^p\frac{\mu(y)}{\mu(E)}\mu(x) \\ 
    &\le 2^{p} (2r)^{p/p'} \sum_{x \in E} \ \sum_{x_E \ge z \ge x}|df(z)|^p\mu(x) \\ 
    &=2^{p} (2r)^{p/p'} \sum_{z \in E}|df(z)|^p \sum_{E \ni x \le z} \mu(x) \\ 
    &\le2^{p} (2r)^{p/p'+1} \sum_{z \in E}|df(z)|^p \mu(z).
\end{align*} In the last line we have used that, for a flow measure, $\sum_{E \ni x \le z}\mu(x) \le \mu(z) \mathrm{diam}(E)$.  Since $|df|\le |\nabla f|$, the above inequalities imply the desired result. 
\end{proof}

\begin{oss}
Observe that a flow measure is a volume measure, in the sense that there exist an edge weight $\omega_{xy}=\omega_{yx}$ that is nonzero if and only if $x \sim y$ and such that
\begin{align*}
    \mu(x)=\sum_{y \in T} \omega_{xy}.
\end{align*}
In the case of a flow, one should choose 
\begin{align*}
    \omega_{xy}=\begin{cases} \frac{1}{2}\min\{\mu(x),\mu(y)\} & x \sim y,\\
    0 &\text{otherwise.}
    \end{cases}
\end{align*}
In the context of weighted graphs, in place of the length of the gradient $|\nabla f|$ it often appears the quantity
\begin{align*}
|\nabla f|_2(x)= \bigg(\sum_{y \sim x} \frac{\omega_{xy}}{\mu(x)}|f(x)-f(y)|^2\bigg)^{1/2},
\end{align*}
which is the right notion of (modulus of the) gradient for the probabilistic Laplacian generated by the transition probability $p(x,y)=\omega_{xy}/\mu(x)$. It is well known (see e.g. \cite{CK}) that $|\nabla f|_2$ and $|\nabla f|$ are comparable quantities if the vertex degree is uniformly bounded and $\inf_{x,y} p(x,y)> 0$. Both the conditions are satisfied from any flow measure which is locally doubling, as a consequence of \cite[Proposition 2.2 and Corollary 2.3]{LSTV}. However, even for non locally doubling flows, it is easy to see that, for any $x\in T$,
\begin{align*}
    |d f(x)| \le 2^{1/2}|\nabla f|_2(x).
\end{align*}
Indeed, 
\begin{align*}
    \frac{1}{2^{1/2}}|f(x)-f(p(x))| \le \frac{1}{2^{1/2}} \bigg(\sum_{y \in s(x)} \frac{\mu(y)}{\mu(x)} |f(x)-f(y)|^2+ |f(x)-f(p(x))|^2 \bigg)^{1/2} =|\nabla f|_2(x).
\end{align*}
It follows that the global Poincaré inequality proved in Theorem \ref{theorem: flow} transfers to the operator $|\nabla|_2$. Namely, for every function $f$ on $T$, every $p\in [1,\infty]$ and every connected subset $E$ of $T$ with $\diam({E})=2r$,
\begin{align*}
    \|f-f_E\|_{L^p(E,\mu)} \lesssim r \| |\nabla f|_2\|_{L^p(E,\mu)}.
\end{align*}
\end{oss}

\bigskip
	
	{\bf{Acknowledgments.}} 
	The authors thank the anonymous referee for useful suggestions which led to an improved version of the paper.\\ 
	 Work partially supported by the MIUR project ``Dipartimenti di Eccellenza 2018-2022" (CUP E11G18000350001) and by the project  ``Harmonic analysis on continuous and discrete structures'' funded by Compagnia di San Paolo (CUP E13C21000270007). The authors are members of the Gruppo Nazionale per l'Analisi Matema\-tica, la Probabilit\`a e le loro Applicazioni (GNAMPA) of the Istituto Nazionale di Alta Matematica (INdAM).
\bibliographystyle{plain}
{\small
\bibliography{references}}

\end{document}